\documentclass[12pt]{amsart}
\usepackage{amssymb,amsmath,mathrsfs}
\usepackage[ps2pdf,colorlinks=true,urlcolor=blue,
citecolor=red,linkcolor=blue,linktocpage,pdfpagelabels,bookmarksnumbered,bookmarksopen]{hyperref}
\usepackage[english]{babel}

\newcommand{\N}{{\mathbb N}}

\newcommand{\R}{{\mathbb R}}

\newcommand{\be}{\begin{equation}}
\newcommand{\ee}{\end{equation}}

\newcommand{\e}{\varepsilon}

\newcommand{\la}{\lambda}

\numberwithin{equation}{section}
\newtheorem{theorem}{Theorem}[section]
\newtheorem{proposition}[theorem]{Proposition}
\newtheorem{corollary}[theorem]{Corollary}
\newtheorem{lemma}[theorem]{Lemma}

\theoremstyle{definition}
\newtheorem{remark}[theorem]{Remark}

\newcommand{\brm}{\begin{remark}\rm}
\newcommand{\erm}{\end{remark}}
\newcommand{\brms}{\begin{remark}\rm}
\newcommand{\erms}{\end{remark}}
\newcommand{\bte}{\begin{theorem}}
\newcommand{\ete}{\end{theorem}}
\newcommand{\bpr}{\begin{proposition}}
\newcommand{\epr}{\end{proposition}}
\newcommand{\ble}{\begin{lemma}}
\newcommand{\ele}{\end{lemma}}
\newcommand{\beq}{\begin{equation}}
\newcommand{\eeq}{\end{equation}}
\newcommand{\bdm}{\begin{displaymath}}
\newcommand{\edm}{\end{displaymath}}
\numberwithin{equation}{section}

\newcommand{\bos}{\begin{remark}\rm}
\newcommand{\eos}{\end{remark}}

\newcommand{\ben}{\begin{enumerate}}
\newcommand{\een}{\end{enumerate}}

\title[Some continuation properties via minimax methods] {Some continuation properties via minimax arguments}

\author[Louis Jeanjean]{Louis Jeanjean}
\address{Laboratoire de Math\'ematiques (UMR 6623)
\newline\indent
Universit\'{e} de Franche-Comt\'{e}
\newline\indent
16, Route de Gray 25030 Besan\c{c}on Cedex, France}
\email{louis.jeanjean@univ-fcomte.fr}

\begin{document}
\subjclass[2000]{}


\begin{abstract}
This note is devotes to some remarks regarding the use of
variational methods, of minimax type, to establish continuity type
results.
\end{abstract}
\maketitle

%\medskip
%\begin{center}
%\begin{minipage}{11cm}
%\footnotesize
%\tableofcontents
%\end{minipage}
%\end{center}

%%%%%%%%%%%%%%%%%%%%%%%%%%%%%%%%%%%%%%
%\medskip

\section{Introduction}
%\linenumbers

The aim of this note is to present some situations where continuity type results can be obtained through the use of minimax type arguments.
\medskip

To give an idea of the type of results we obtain let us first consider the
equation
\begin{equation}\label{0.1}
- \Delta u + \lambda u = V(x) g(u), \quad u \in H^1(\R^N).
\end{equation}
Here $\lambda >0$, $V$ is radially symmetric and $g$ is assumed to be a nonlinear term, superlinear at the origin and
subcritical under which~\eqref{0.1} has a non
trivial positive solution. Then assuming that, for any fixed $\lambda
>0$ equation~\eqref{0.1} has at most one positive solution we 
prove that the map $\lambda \to u_{\lambda} \in H^1(\R^N)$ is
continuous. Namely we establish the existence of a global branch
of solutions.
\medskip

As a second example consider the equation
\begin{equation}\label{0.2}
- \Delta u = |u|^{p-1}u + f(x), \quad u \in H^1_0(\Omega)
\end{equation}
where $\Omega \subset \R^N$ is an open regular bounded domain, $1 < p < \frac{N+2}{N-2}$ and $f \in L^q(\Omega)$ for some
$q > \frac{N}{2}$. We show that there exists a $\alpha >0$ such
that if $||f||_{q} \leq \alpha$ then (\ref{0.2}) admits a positive
solution on $\Omega$. When $||f||_q$ is small enough it is standard to show that there exists a solution. If $f \geq
0$, using the maximum principle, it follows that it is positive. Here we prove, without assumption on the sign of $f$, but possibly decreasing the value of $||f||_q$, that this is still true.

The note is organized as follows. In Section \ref{Section2} we present some abstract considerations. In Section \ref{Section3} we apply them to a problem of the type of~\eqref{0.1}. Section \ref{Section4} deals with the nonhomogeneous problem~\eqref{0.2}.

\section{Some abstract considerations}\label{Section2}

Let $X$ be a reflexive Banach space whose norm is denoted $||\cdot||$.  Consider for some $\e >0$ and $\la \in ]1- \e, 1 + \e[$ a familly $(I_{\la})$ of $C^1$ functionals on $X$ of the form
$$I_{\la}(u) = A(u) - \la B(u), \quad \la \in ]1- \e, 1 + \e[.$$
We assume that
\begin{itemize}
\item[(A1)] Both $B \in C^1(X, \R)$ and its derivative $B' \in C(X, \R)$ take bounded sets to bounded sets. \smallskip
\item[(A2)] For any $\lambda \in ]1- \e, 1 + \e[ \backslash \{1\}$ $I_{\la}$ has a critical point $u_{\la}$ at a level denoted $c_{\la}$. Moreover there exists a bounded interval $S \subset \R$ such that $c_{\lambda} \in S$ if
$\lambda \in ]1- \varepsilon, 1+ \varepsilon[ \backslash \{1\}$.\smallskip
\item[(A3)] For any sequence  $(\la_n) \subset ]1 - \e, 1 + \e[\backslash \{1\} $ with $\lambda_n \to 1$ the sequence
$(u_{\lambda_n}) \subset X $ is bounded. \smallskip
\item[(A4)] Any bounded Palais-Smale sequence $(v_n)$ for $I := I_1$ such
that $(I(v_n)) \subset S$ admits a converging subsequence.
\end{itemize}
Under these assumptions we have :

\begin{theorem}\label{thabstract}
Assume that (A1)-(A4) hold. Then 
\begin{itemize}
\item[(i)] There exists a critical point $u$ of  $I$ such that $I(u) \in S$. \smallskip
\item[(ii)] Any sequence $(u_{\la_n})$ with  $\la_n \in ]1 - \e, 1 + \e [  \backslash \{1\}$ and $\la_n \to 1$ converges, up to a subsequence, toward a critical point of $I$, associated to a level in $S$.
\end{itemize}
\end{theorem}

\begin{proof} First we prove (i). Let $(\la_n) \subset ]1 - \e, 1 + \e [  \backslash \{1\}$ satisfies $\la_n \to 1$. By (A3) the sequence $(u_n):= (u_{\la_n})$ is bounded. Now we have
$$I(u_n) = I_{\lambda_n}(u_n) + (\lambda_n -1) B(u_n)$$
$$I'(u_n)= I'_{\lambda_n}(u_n) + (\lambda_n -1) B'(u_n).$$
By (A1) it follows that
$$(\lambda_n -1) B(u_n) \to 0 \, \quad \mbox{ and } \quad (\lambda_n -1) B'(u_n) \to
0.$$ Thus
$$ I(u_n) = c_{\lambda_n} + o(1) \quad \mbox{ and } \quad I'(u_n) = o(1).$$
Since, by (A2), the sequence $(c_{\lambda_n}) \subset S$ is
bounded, it follows that $(u_{\lambda_n})$ is a (bounded)
Palais-Smale sequence for $I$. By (A4) we then know that, passing to a subsequence,
$u_n \to u$ with $u \in X$ a critical point of $I$ associated to a
value in $S$. This proves (i). Now (ii) follows from the proof of (i).  
\end{proof}

\begin{corollary}\label{coabstract}
Assume that (A1)-(A4) hold and that for $\lambda =1$ there exists at most a critical point $u \in X$ corresponding to a value in $S$. Then as $\la \to 1$ we have $u_{\la} \to u$. In particular 
$c_{\lambda} \to c$, where $c = I(u)$.

\end{corollary}

\begin{proof}
If we assume by contradiction that $u_{\lambda}$ do not converge
toward $u$ if $\lambda \to 1$ then on one hand there exists a sequence
$(\la_n) \subset ]1 - \e, 1 + \e [  \backslash \{1\}$ with $\lambda_n \to 1$ and a  $\delta >0$ such that
$||u_{\lambda_n} -u || \geq \delta
>0.$ On the other hand repeating the proof of Theorem~\ref{thabstract} on the sequence
$(u_{\lambda_n})$ we arrive at the conclusion that, up to a
subsequence, $u_{\lambda_n} \to u$. This is a  contradiction.
\end{proof}

\begin{remark}
Using the results of \cite{Je2} it can be shown, under very general assumptions on the family $I_{\la}$, that for almost any $\la \in ]1- \e, 1 + \e [$, $I_{\la}$ admit a bounded Palais-Smale sequence whose value stays within a compact. So if for any $\la \in ]1 - \e, 1 + \e [$ any bounded Palais-Smale sequence for $I_{\la}$ admit a converging subsequence we see that assumption (A2) holds.
\end{remark}

\section{Problems on $\R^N$}\label{Section3}

\subsection{Autonomous cases}
We consider the equation
\begin{equation}\label{3.01}
- \Delta u + \lambda u = g(u), \quad u \in H^1(\R^N), \quad  N \geq 3
\end{equation}
where we assume that $g \in C(\R, \R)$ satisfies
\begin{itemize}
\item[(H1)]  $g(s) /s \to 0$ as $s\to 0$. \smallskip
\item[(H2)]  For some $p \in ]1, \frac{N+2}{N-2}[$ 
$$g(s)/ |s|^p \to 0 \quad \mbox{as } s \to \infty.$$ 
\item[(H3)] There exists $s_0 >0$ such that $G(s_0) >0$ with $G(s) :=
\int_0^s g(t)dt.$ \smallskip
\end{itemize}
The natural functional associated to (\ref{3.01})
is defined on $H:=H^1(\R^N)$ by
$$I_{\la}(u) = \frac{1}{2}\int_{\R^N}|\nabla u|^2 + \lambda |u|^2 dx -
\int_{\R^N}G(u) \thinspace dx.$$
It is standard that under (H1)-(H2) $I_{\lambda}$ is well
defined and of class $C^1$. We define the least energy level
\begin{equation}\label{3.2}
m_{\lambda}= inf\{ I_{\lambda}(u), u \in H \backslash \{0\},
I_{\la}'(u)=0 \}
\end{equation}
and the set of least energy solution
$$G_{\lambda}= \{ u \in H \backslash \{0\}, I_{\la}'(u)=0, I_{\la}(u) = m_{\lambda}\}.$$
\vskip4pt Also let $\lambda^* = sup\{ \lambda >0$ : $\exists \nu
>0$ such that $ G(\nu) - \lambda /2 \nu^2 >0 \}$. \medskip

From~\cite{BeLi} it is known that, under
(H1)-(H3) and for any $0 < \lambda < \lambda^* $, $m_{\lambda}
>0$ and $G_{\lambda} \neq \emptyset$ contains a positive element. In addition it is shown in~\cite{JeTa1} that $m_{\lambda}$ admits a mountain pass
characterization. Namely that
\begin{equation}\label{defmin}
m_{\la} := \inf_{g \in \Gamma_{\la}} \max_{t \in [0,1]} I_{\la}(g(t))
\end{equation}
where 
$$\Gamma_{\la}:= \{ g \in C([0,1], H), g(0) = 0, I_{\la}(g(1)) <0\}.$$
This characterization implies that $\la \to m_{\la}$ is nondecreasing.
Lastly we known from~\cite{ByJeMa} that any
element of $G_{\lambda}$ is radially symmetric and have a given sign. We prove the
following result.

\begin{theorem}\label{thautonomous}
Assume that (H1)-(H3) hold. Let $\lambda_0 \in ]0, \lambda^*[$ and \\
$\{(\lambda_n, u_n)\} \subset ]0,  \lambda^*[ \times H$ be a sequence
such that
$$u_n \in G_{\lambda_n} \, \mbox{ and } \, 
\lambda_n \to \lambda_0 \, \mbox{ as } \, n \to \infty.$$ Then there
exist a $u_0 \in G_{\la_0}$ and a subsequence $(u_{n_k})$
of $(u_n)$ such that
$$u_{n_k} \to u_0 \, \mbox{ as } \, k \to \infty.$$
In particular $\la \to m_{\la}$ is continuous.
\end{theorem}
The proof of Theorem \ref{thautonomous} will follow from three lemmas. Since any $G_{\la}$  only contains radially symmetric functions we can, without restriction, work in the subspace 
$$H^1_r(\R^N):=\{ u \in H^1(\R^N), u(x) = u(|x|) \}.$$
We also recall that any critical point of $I_{\la}$ in $H^1_r(\R^N)$ is, by the principle of symmetric criticality of Palais, also a critical point on all $H^1(\R^N)$.
We set $H_r := H_r^1(\R^N)$.

\begin{lemma}\label{bounded-au}
Under the assumptions of Theorem \ref{thautonomous} the sequence $(u_n) \subset H_r$ is bounded.
\end{lemma}

\begin{proof}
Since $u_n, n \in \N$ is a critical point of $I_{\la_n}$ we know from \cite{BeLi} that it satisfies the Pohozaev identity
$$(N-2) ||\nabla u_n||^2_2 = 2N \Big[ - \frac{\la_n}{2} ||u_n||_2^2 + \int_{\R^N} G(u_n) dx \Big]$$
and thus we have
\begin{equation} \label{auto}
I_{\la_n}(u_n) = \frac{1}{N} ||\nabla u_n||^2_2.
\end{equation}
By the mountain pass characterization (\ref{defmin}) the function $\la \to m_{\la}$ is non decreasing and since $I_{\la_n}(u_n) = m_{\la_n}$ we deduce from (\ref{auto}) that $(||\nabla u_n||^2_2) \subset \R$ is bounded. Now since $I_{\la_n}'(u_n)u_n = 0$ we have
\begin{equation}\label{3.1}
\int_{\R^N} |\nabla u_n|^2 + \la_n |u_n|^2 \thinspace dx = \int_{\R^N}g(u_n)u_n \thinspace dx.
\end{equation}
By (H1)-(H2) for any $\delta >0$ there exists $C_{\delta} >0$ such that
$$|g(s)| \leq \delta |s| + C_{\delta} |s|^{\frac{N+2}{N-2}} \quad \mbox{for all } s \in \R.$$
Thus using the Sobolev embeddings, it follows from (\ref{3.1}) that, for a $C>0$,
$$\la_n \int_{\R^N} |u_n|^2 \thinspace dx \leq \delta \int_{\R^N} |u_n|^2 \thinspace dx + C_{\delta} C ||\nabla u_n||_2^{\frac{2N}{N-2}}.$$
Since $\lambda_n \to \la_0 >0$, choosing $\delta >0$ sufficiently small and using the fact that $(||\nabla u_n||_2) \subset \R$ is bounded we see that
$(||u_n||_2) \subset \R$ is also bounded.
\end{proof}

\begin{lemma}\label{PS-hold-au}
Under the assumptions of Theorem \ref{thautonomous} any bounded Palais-Smale sequence for $I_{\lambda_0}$ admits a
converging subsequence.
\end{lemma}

\begin{proof}
Let $(u_n) \subset H_r$ be a bounded Palais-Smale sequence for $I_{\la_0}$. Since $(u_n) \subset H_r$ is bounded we have that $u_n \rightharpoonup u$ in $H$ and $u_n \to u$ in $L^p(\R^N)$ for $p \in ]2, \frac{2N}{N-2}[$ (see \cite{St}). Now since $(u_n) \subset H_r$ is a Palais-Smale sequence we have, in the dual $H_r^{-1}$,
$$- \Delta u_n + \lambda_n u_n - g(u_n) \to 0.$$
Using the strong convergence of $(u_n)$ we readily deduce from (H1)-(H2) that $g(u_n) \to g(u)$ in $H_r^{-1}$. Thus
\begin{equation}\label{eq:3.2}
- \Delta u_n + \la u_n \to g(u) \, \mbox{ in } \, H_r^{-1}.
\end{equation}
Now let $L: H_r \to H_r^{-1}$ be defined by
$$(Lu)v = \int_{\Omega} \nabla u \nabla v + \lambda uv \thinspace dx.$$
This operator is invertible and so we deduce from (\ref{eq:3.2}) that
$$u_n \to L^{-1}(G(u)) \, \mbox{ in } \, H_r.$$
Consequently by the uniqueness of the limit $u_n \to u$ in $H_r$.
\end{proof}

\begin{lemma}\label{PS-hold-initial-au}
The sequence $(u_n) \subset H_r$ is a, bounded, Palais-Smale
sequence for $I_{\lambda_0}$ at the level $m_{\lambda_0}$.
\end{lemma}

\begin{proof}
We already know that the sequence $(m_{\la_n}) \subset \R$ is bounded. Now we have 
$$I_{\lambda_0}(u_n) = I_{\lambda_n}(u_n) + (\lambda_n - \lambda_0)
||u_n||^2_2$$
$$I_{\lambda_0}'(u_n) = I_{\lambda_n}'(u_n) + (\lambda_n -
\lambda_0) u_n.$$ 
Thus, since $I_{\lambda_n}'(u_n) =0$ and $(u_n)
\subset H_r$ is bounded we have $I'_{\lambda_0}(u_n) \to 0.$ Also
$(\lambda_n - \lambda_0) ||u_n||^2_2 \to 0$ and we deduce that
$(I_{\lambda_0}(u_n)) \subset \R$ is bounded. This proves that
$(u_n) \subset H_r$ is a bounded Palais-Smale sequence for
$I_{\lambda_0}$. From Lemma~\ref{PS-hold-au} we deduce that $u_n
\to u_0$ with $u_0 \in H_r$ a critical point of
$I_{\lambda_0}$. To conclude we just need to prove that
$m_{\lambda_n} \to m_{\lambda_0}.$ But, by the convergence $u_n
\to u_0$, we have that
\begin{equation}\label{up}
\lim_{n \to \infty} m_{\lambda_n} = \lim_{n \to \infty}I_{\la_n}(u_n) = I_{\la_0}(u_0) \geq m_{\lambda_0}.
\end{equation}
Now considering a sequence $(\la_n)$ increasing to $\la_0$, and using the fact that $\lambda \to
m_{\lambda}$ is non decreasing, we deduce from (\ref{up}) that  $m_{\lambda_n} \to m_{\lambda_0}.$ 
\end{proof}

\begin{proof}[Proof of Theorem~\ref{thautonomous}]
Gathering Lemmas~\ref{bounded-au},~\ref{PS-hold-au}
and~\ref{PS-hold-initial-au} we immediately conclude. 
\end{proof}
From Theorem~\ref{thautonomous} we deduce
\begin{corollary}\label{coautonomous}
Assume that (H1)-(H3) hold and that, for any $\lambda \in ]0,
\lambda^*[$ the set $G_{\lambda}$ contains only one positive
element $u_{\lambda}$. Then the map $\lambda \to u_{\lambda}$ is
continuous from $]0, \lambda^*[$ to $H_r$.
\end{corollary}

\begin{proof}
Let $\lambda_0 >0$ be fixed and assume, by contradiction, that
there exist a sequence $(\lambda_n) \subset ]0, \lambda^*[$ with
$\lambda_n \to \lambda_0$ and a $\delta >0$ such that
$||u_{\lambda_n} - u_{\lambda_0}|| \geq \delta.$ Then we deduce
from Theorem~\ref{thautonomous} that $u_{\lambda_n} \to u_0 \in
G_{\lambda_0}.$ Since the nonnegative property is preserved by the convergence using the maximum principle we obtain that $u_0$ is positive and thus by uniqueness $u_0 =
u_{\lambda_0}$.
\end{proof}

\begin{remark}\label{uniqueness}
The problem of deriving conditions on $g$ which insure that~\eqref{3.01}
has a unique positive solution (and thus an unique positive ground state) has been extensively studied. The uniqueness is known to hold for a large class of nonlinearities. See for example \cite{PuSe} and the references therein in that direction.
\end{remark}

\begin{remark}\label{shatah-strauss}
A result essentially the same as Corollary \ref{coautonomous} was previously obtained in \cite{ShSt} (see also \cite{Sh}). Let us also point out that another proof of Theorem~\ref{thautonomous} follows from Proposition 5.5 of \cite{JeTa2}. The proofs that we give here, are different and we believe somehow simpler.
\end{remark}

\subsection{Non-autonomous cases}\label{non-autonomous}

We consider now the equation
\begin{equation}\label{33.1}
- \Delta u + \lambda u = V(x) g(u), \quad u \in H^1(\R^N)
\end{equation}
where we assume that $g \in C(\R, \R)$ satisfies in addition to
(H1)-(H2) 
\begin{itemize}
\item[(H4)] There exists $\mu >2$ such that
$$ 0 < \mu G(s) \leq g(s)s, \quad \forall s \in \R \, \mbox{ with } \, G(s):= \int_0^S g(t) dt.$$
\end{itemize}
On the potential $V \in C(\R^N, \R)$ we assume  \smallskip
\begin{itemize}
\item[(V)]  $V \geq 0, V \neq 0$ and either $V$ is radial or $\lim_{|x| \to \infty}V(x) = 0.$ \smallskip
\end{itemize}

Under (H1),(H2),(H4) and (V) it is standard to show
that~\eqref{33.1} admit, for any $\lambda
>0$, a non trivial solution as a critical point of the $C^1$ functional
$$ I_{\la}(u) = \frac{1}{2}\int_{\R^N}|\nabla u|^2 + \la |u|^2 \thinspace dx - \int_{\R^N} V(x) G(u) \thinspace dx.$$
Indeed, one just need to use the mountain pass theorem
(see~\cite{AmRa}). The boundedness of Palais-Smale sequence
follows from (H4) and because of (V) any bounded Palais-Smale sequence admits a converging subsequence. Thus one obtain a critical point at the mountain pass level that we denote $c_{\la} >0$. We now assume  \smallskip
\begin{itemize}
\item[(U)] For any $\lambda >0$,~\eqref{33.1} admits at most one positive
solution that we denote $u_{\lambda} \in H^1(\R^N)$. \smallskip
\end{itemize}

Without restriction (by a suitable modification of $g$ for $s <0)$ we can assume that the mountain pass solution is positive and thus that it coincide with $u_{\la}$. \medskip

Our result is the following
\begin{theorem}\label{thnonautonomous}
Assume that (H1),(H2),(H4) and (V), (U) hold. Then the map $\lambda
\to u_{\lambda}$ from $]0, + \infty[$ to $H^1(\R^N)$ is
continuous.
\end{theorem}

\begin{proof}
The proof follows closely the one of Theorem \ref{thautonomous} in the autonomous case. Our working space $H$ is $H^1(\R^N)$ if $\lim_{|x| \to \infty}V(x)
=0$ and $H^1_r(\R^N)$ if $V$ is radial. We assume, by
contradiction, that there exists a $\la_0 >0$, a sequence $(\lambda_n) \subset
]0, + \infty[$ with $\lambda_n \to \lambda_0$ and a $\delta >0$
such that $||u_{\lambda_n} - u_{\lambda_0}|| \geq \delta.$ First we show that the sequence $(u_n) \subset H$ is bounded. We have
$$I_{\lambda}(u) = \frac{1}{2}\int_{\R^N} |\nabla u|^2 + \lambda
|u|^2 dx - \int_{\R^N} V(x) G(u) \thinspace dx = c_{\lambda}.$$
$$ I'_{\lambda}(u)u = \int_{\R^N} |\nabla u|^2 + \lambda
|u|^2 dx - \int_{\R^N} V(x) g(u)u \thinspace dx = 0.$$ Thus, using (H4),
\begin{align*}\label{ex34}
      \frac{1}{2}\int_{\R^N} |\nabla u|^2 + \lambda
|u|^2 \thinspace dx    & = c_{\lambda} +  \int_{\R^N} V(x) G(u) \thinspace dx\\
& \leq c_{\lambda} + \frac{1}{\mu} \int_{\R^N} V(x)
g(u) u \thinspace dx  \\
\nonumber &  \leq c_{\lambda} + \frac{1}{\mu}\int_{\R^N} |\nabla
u|^2 + \lambda |u|^2 dx.
   \end{align*}
Hence $$\left(\frac{1}{2}- \frac{1}{\mu}\right) \int_{\R^N} |\nabla u|^2 +
\lambda |u|^2 dx \leq c_{\lambda}$$
and $(u_n) \subset H$ is indeed bounded. Now reasoning as in Lemma \ref{PS-hold-initial-au} we deduce that the sequence $(u_n) \subset H$ is a, bounded, Palais-Smale sequence for $I_{\la_0}$. Finally, following the proof of Lemma \ref{PS-hold-au} we can show that any bounded Palais-Smale sequence for $I_{\la_0}$ admit a converging subsequence. At this point, using the uniqueness, we deduce that $u_{\la_n} \to u_{\la_0}$ and this contradiction concludes the proof. 
\end{proof}

\begin{remark}\label{key}
The key feature that guarantee that the results presented in this Section hold is the fact that the mountain pass level coincides with the least energy level. In Theorem~\ref{thautonomous} it follows from the result of \cite{JeTa1} and in Theorem \ref{thnonautonomous} by the uniqueness of positive solutions. 
\end{remark}

\begin{remark}\label{genoud2}
In~\cite{Ge}, see Proposition 1, the analog of Theorem \ref{thautonomous} is establish for (\ref{33.1}) in the case where $g(u) = |u|^{p-1}u$ for $p \in ]1,
\frac{N+2}{N-2}[$. In addition when the uniqueness of $u_{\la}$ is assumed Theorem \ref{thnonautonomous} holds true. The proofs given in~\cite{Ge} use strongly the existence of a well defined Nehari manifold for $I_{\la}$. Using this manifold is possible when the function $s \to g(s)/s$ is stricly increasing. Under this condition it is now standard, see for example Lemma 1.2 in \cite{DeFe} or Proposition 3.11 in \cite{Ra}, that the mountain pass level $c_{\la}$ coincides with the least energy level. Thus we can recover and extend to the results of \cite{Ge} using the approach developped in Theorem~\ref{thautonomous}.  
\end{remark}

\begin{remark}
In equation (\ref{3.01}) we have restricted ourselves to $N \geq 3$. The case $N=2$ can also be treated under the assumption that $p \in ]1, + \infty[$. Some additional work however is necessary at the level of Lemma \ref{bounded-au} to show the boundedness of the sequence $(u_n) \subset H$. See \cite{JeTa2} in that direction.
\end{remark}

\begin{remark}
The only purpose of condition (H4) is to insure the boundedness of Palais-Smale sequence for $I_{\la}$ (or of sequences of critical points of $I_{\la_n}$). Alternative conditions are possible.
\end{remark}

\section{On a non-homogeneous problem}\label{Section4}

We consider here
\begin{equation} \label{4.1}
- \Delta u = |u|^{p-1}u + f(x), \quad u \in H_0^1(\Omega)
\end{equation}
where $\Omega \subset \R^N, N \geq 3$ is an open bounded regular
domain, $1 < p < \frac{N+2}{N-2}$ and $f \in L^{q} (\Omega)$ for
some $q > \frac{N}{2}$. We prove the following result

\begin{theorem}\label{th:posremain}
Under the assumptions above there exists a $\alpha >0$ such that when $||f||_{q}
\leq \alpha$ the equation (\ref{4.1}) admits a positive solution on
$\Omega$.
\end{theorem}

We set $g(s) = s^p$ if $s \geq 0$, $g(s) = 0$ if $s \leq 0$ and $G(s) = \int_0^s g(t) \thinspace dt.$
The functional associated to (\ref{4.1}) is defined on $H:= H^1_0(\Omega)$ by
$$I_f(u) = \frac{1}{2}\int_{\Omega}|\nabla u|^2 dx -
\int_{\Omega}G(u) dx - \int_{\Omega}fu \thinspace dx.$$
Under our assumptions it is standard to show that $I \in C^1(H, \R)$. We shall work with the norm $||u|| := ||\nabla u||$ on $H$.
\begin{remark}  The point of Theorem \ref{th:posremain} is the existence of a
positive solution on $\Omega$ without assuming a sign on $f$. If $f \geq 0$  the result follows directly from the existence of a critical point for $I_f$. 
\end{remark}

\begin{lemma}\label{le:mp}
Under the assumptions of Theorem \ref{th:posremain} there exists a $\beta >0$ and a $\gamma >0$ such that for any $f$ satisfying $||f||_q \leq \beta$ the functional $I_f$ has a critical point $u_f$ at a value $c_f \geq \gamma >0$.
\end{lemma}

\begin{proof}
Let $||f||_q \leq \beta $ with $\beta >0$ to be determined later. First we show that the functional $I_f$ has a mountain pass geometry in the sense that $I_f(0) = 0$ and
\begin{itemize}
\item[(i)] There exist $a>0$, $b >0$ such that if $||u|| = a$ then $I_f(u) \geq b.$ \smallskip
\item[(ii)] There exists a $v \in H$ with $||v|| >a$ such that $I_f(v) \leq 0$. \smallskip
\end{itemize}
To prove (i) observe that, by Holder and Sobolev embeddings, for $1/q + 1/q' =1$ and a $C>0$,
\begin{align}
I_f(u) & = \frac{1}{2}|| u||^2 - \frac{1}{p+1}||u^+||_{p+1}^{p+1} - ||f||_q ||u||_{q'}  \\
\nonumber & \geq \frac{1}{2}|| u||^2  - C ||u||^{p+1} - C ||f||_q ||u||.
\end{align}
We first fix $a >0$ sufficiently small so that
$$\frac{1}{2}|| u||^2  - C ||u||^{p+1} \geq \frac{1}{4}||u||^2 \quad \mbox{if} \quad ||u||=a.$$
Then we choose $\beta >0$ sufficiently small so that $ C \, ||f||_q \, a \leq \frac{1}{8} a^2.$ At this point $(i)$ hold. 
To show (ii) it suffices to observe that taking a $u \in H$ with $u >0$ on $\Omega$ one has $I_f(tu) \to - \infty$ as $t \to + \infty.$  \medskip

Next we show that the Palais-Smale condition holds, namely that any Palais-Smale sequence admits a convergent subsequence. Let $(u_n) \subset H$ be a Palais-Smale sequence for $I_f$ at a level $c_f \in \R$. We have, for $n \in \N$ large enough,
\begin{align}\label{4.10}
c_f +1 + ||u_n|| & \geq  I_f(u_n) - \frac{1}{p+1}I_{f}'(u_n)u_n   \\
\nonumber & =  \Big(\frac{1}{2}- \frac{1}{p+1}\Big) ||u_n||^2 - \Big(1-\frac{1}{p+1}\Big) \int_{\Omega} fu_n \thinspace dx \\
\nonumber & \geq \Big(\frac{1}{2}- \frac{1}{p+1}\Big) ||u_n||^2 -\frac{p}{p+1} ||f||_{H^{-1}}||u_n||
\end{align}
and thus $(u_n) \subset H$ is indeed bounded. The fact that $(u_n) \subset H$ admit a convergent subsequence is now standard since we work on a bounded domain. At this point the assumption of the mountain pass theorem, see \cite{AmRa}, are satisfied and the lemma follows.
\end{proof}

\begin{proof}[Proof of Theorem \ref{th:posremain}]
In view of Lemma \ref{le:mp} it just remains to show, by possibly decreasing the value of $\beta >0$, that the solution $u_f \in H$ is positive on $\Omega$. For this we consider the limit problem
\begin{equation}\label{2}
 - \Delta u = g(u).
\end{equation}
The functional associated to  (\ref{2}) is
$$ I(u) =\frac{1}{2}\int_{\Omega}|\nabla u|^2 dx -
\int_{\Omega}G(u) dx. $$ 
Clearly, by construction of $g$, any critical point of $I$ is nonnegative. Moreover if $u \in H$ is a non trivial critical point of $I$, by the Hoft maximum principle we obtain that $u >0$ on $\Omega$ and
also that its normal derivatives are strictly positive on $\partial \Omega$. \medskip

Now we assume, by contradiction, that there exists a sequence $(f_n) \subset L^q(\Omega)$ such that $f_n \to 0$ in $L^q(\Omega)$ for which the corresponding sequence $u_n := u_{f_n}$ remains non positive for any $n \in \N$. Clearly we shall reach a contradiction if we manage to show that, up to a subsequence $u_n \to u$ where 
$u \in H$ is a non trivial solution for the problem
(\ref{2}). Indeed starting from $u_n \to u$
in $H$, by standard elliptic regularity estimates, it
follows that $u_n \to u$ in $C^1(\bar{\Omega})$. 
\medskip

To show this we first observe that $(u_n) \subset H$ is bounded. Indeed, from (\ref{4.10})  where $f$ is replaced by $f_n$ it is the case if $(c_{f_n}) \subset \R^+$ remains bounded. But taking any fixed $u \in H$ with $u >0$ we have, for a $C>0$,
\begin{align}\label{boundPS}
c_{f_n} & \leq  \max_{t >0} I_{f_n}(tu)  \\
\nonumber & \leq   \frac{1}{2} t^2 ||u||^2 - \frac{t^p}{p+1} ||u||^{p+1}_{p+1} + t \int_{\Omega} |f_n| u \thinspace dx  
\end{align}
and the bound on $(c_{f_n}) \subset \R^+$ follows. Now when $n \to + \infty$,
$$I(u_n) = I_{f_n}(u_n) + \int_{\Omega}f_nu_n \thinspace dx$$ remains bounded
since $(c_{f_n}) \subset \R^+$ is bounded and $\int_{\Omega}f_nu_n dx \to 0$ (because $(u_n) \subset H$ is bounded). Also
$$I'(u_n) = I'_{f_n}(u_n) + f_n \to 0 \quad \mbox{ in } H^{-1}.$$
Thus $(u_n) \subset H$ is a bounded Palais-Smale sequence for $I$.
To conclude we observe that since $(u_n) \subset H$ is a bounded Palais-Smale sequence for $I$ it converges strongly,  up to a subsequence, in
$H$ towards a non trivial critical point of $I$. The fact that it is non trivial follows from the estimate $c_{f_n} \geq \gamma >0. $ This ends the proof.
\end{proof}

\begin{remark}\label{r:1}
Clearly under the assumptions of Theorem \ref{th:posremain}, equation (\ref{th:posremain}) admits a second solution as a local minimum of $I_f$. But when $f \to 0$ this solution converges to $0$ and there is no reason for it to be positive.
\end{remark}

\begin{remark}\label{r:2}
A numerous literature (see for example \cite{AdTa, Je}) is devoted to the problem of finding two solutions or more, for equations of the form
\begin{equation}\label{onR}
- \Delta u + u = a(x) |u|^{p-1}u + f(x), \quad u \in H^1(\R^N).
\end{equation}
So far, up to our knowledge, multiple solutions are obtained only under the assumption that $f\geq 0$ on $\R^N$ (and $||f||_{H^{-1}}$ small enough). An interesting question would be to study if for problems of the type of (\ref{onR}) a multiplicity result can be obtained without requiring $f$ to be non negative. In that direction we suspect that the approach followed in Theorem \ref{th:posremain} could proved useful.
\end{remark}

\noindent {\bf Acknowledgements:} The author would like to thanks Professors M. Ohta, C.A. Stuart, K. Tanaka and Doctor F. Genoud for useful remarks on a preliminary version of this work.


\vskip20pt

\begin{thebibliography}{99}

\bibitem{AdTa}
{\sc S.\ Adachi and K. Tanaka}, Four positive solutions for the semilinear elliptic equation : $- \Delta u + u = a(x) u^p + f(x) \in \R^N$,
{\em Calc. Var. Partial Differential Equations},  {\bf 11}, (2000), 63-95.

\bibitem{AmRa}
{\sc A.\ Ambrosetti and P.H. Rabinowitz}, Dual variational methods
in critical point theory and applications, {\em J. Functional
Analysis},  {\bf 14}, (1973), 349-381.

\bibitem{BeLi}
{\sc H.~Berestycki, P.L.~Lions}, Nonlinear scalar field equations
I, {\em Arch. Rat. Mech. Anal. } \textbf{82} (1983), 313--346.

\bibitem{ByJeMa}
{\sc J.~Byeon, L.~Jeanjean, M.~Maris}, Symmetry and monotonicity
of least energy solutions, {\em Calc. Var. Partial Differential
Equations}, \textbf{36}, (2009), 481-492.

\bibitem{DeFe}
{\sc  M. del Pino and P. Felmer,} { Local mountain passes for semilinear elliptic problems in unbounded domains,}
 {Calculus of Variations and PDE}, {\bf 4}, (1996), 121-137.

\bibitem{Ge}
{\sc F.\ Genoud}, A smooth global branch of solutions for a
semilinear elliptic equation on $\R^N$, {\em Calc. Var. Partial
Differential Equations}, {\bf 38}, (2010), 207-232.

\bibitem{Je}
{\sc L.\ Jeanjean}, Two positive solutions for a class of
nonhomogeneous elliptic equations, {\em Diff. Int. Equations},
{\bf 10}, (1997), 609-624.

\bibitem{Je2}
{\sc L.\ Jeanjean}, On the existence of bounded Palais-Smale
sequences and application to a Landesman-Lazer-type problem set on
$\R^N$, {\em Proc. Roy. Soc. Edinburgh}, {\bf 129A}, (1999),
787-809.

\bibitem{JeTa1}
{\sc  L. Jeanjean and K. Tanaka,} A remark on least energy solutions in
$\R^N$, {\em  Proc Amer. Math. Soc.} {\bf 131}, (2003), 2399-2408.

\bibitem{JeTa2}
{\sc  L. Jeanjean and K. Tanaka,} { Singularly perturbed elliptic problems with superlinear and asymptotically linear linearities,}
 {Calculus of Variations and PDE}, {\bf 21}, (2004), 287-318.


\bibitem{PuSe}
{\sc  P. Pucci and J. Serrin,} {Uniqueness of ground states for quasilinear elliptic operators,} {Indiana Univ. Math. J.}, {\bf 47},
(1998), 501-528.

\bibitem{Ra}
{\sc  P.H. Rabinowitz,} {On a class of nonlinear Schrodinger equations,} {Z. Angew. Math. Phys.}, {\bf 43},
(1992), 270-291.

\bibitem{Sh}
{\sc  J. Shatah,} {Stable standing waves of nonlinear Klein-Gordon equations,} {Comm. Math. Phys.}, {\bf 91},
(1983), 313-327.

\bibitem{ShSt}
{\sc  J. Shatah and W.A. Strauss} {Instability of nonlinear bound states,} {Comm. Math. Phys.}, {\bf 100},
(1985), 173-190.

\bibitem{St}
{\sc  W.A. Strauss,} {Existence of solitary waves in higher dimensions} {Comm. Math. Phys.}, {\bf 55},
(1977), 149-162.




\end{thebibliography}
\end{document}